\documentclass[12pt,a4paper]{article}

\usepackage{latexsym}
\usepackage{amsmath}
\usepackage{amssymb}
\usepackage{arydshln}

\usepackage{cite}
\usepackage{stmaryrd}
\usepackage{enumerate}

\usepackage{hyperref}

\usepackage{color}
\usepackage{lineno}
\usepackage{graphicx}
\usepackage{ae}
\usepackage{amsmath}
\usepackage{amssymb}
\usepackage{latexsym}
\usepackage{url}
\usepackage{epsfig}
\usepackage{mathrsfs}
\usepackage{amsfonts}
\usepackage{amsthm}

\usepackage{float}
\usepackage{subfig}
\usepackage{tikz}
\usetikzlibrary{positioning}

\newtheorem{theorem}{Theorem}[section]

\newtheorem{lemma}[theorem]{Lemma}

\theoremstyle{definition}

\numberwithin{equation}{section} 
\allowdisplaybreaks    

\def\qed{\hfill$\Box$\vspace{12pt}}

\long\def\delete#1{}

\usetikzlibrary{decorations.markings}

\tikzstyle{vertex}=[circle, draw, inner sep=0pt, minimum size=6pt]
\tikzstyle{directed}=[postaction={decorate,
	decoration={markings,mark=at position 0.3 with {\arrow{stealth}}}
}]

\usepackage{xcolor}
\usepackage[normalem]{ulem}

\begin{document}
\title {Maximum index of unbalanced signed graphs without a positive odd cycle}

\author{Li Ma$^{a,b}$,~Linfeng Xie$^{a,c}$,~Xiaogang Liu$^{a,c,}$\thanks{Supported by the National Natural Science Foundation of China (No. 12371358).}~$^,$\thanks{ Corresponding author. Email addresses: malixj@mail.nwpu.edu.cn, xielinfeng@mail.nwpu.edu.cn, xiaogliu@nwpu.edu.cn}~
	\\[2mm]
	{\small $^a$School of Mathematics and Statistics,}\\[-0.8ex]
	{\small Northwestern Polytechnical University, Xi'an, Shaanxi 710072, P.R.~China}\\
	{\small $^b$College of Mathematics and Physics,}\\[-0.8ex]
	{\small Xinjiang Agricultural University, Urumqi, Xinjiang 830052, P.R.~China}\\
	{\small $^c$Xi'an-Budapest Joint Research Center for Combinatorics,}\\[-0.8ex]
	{\small Northwestern Polytechnical University, Xi'an, Shaanxi 710129, P.R. China}\\
}
\date{}

\openup 0.5\jot
\maketitle

\begin{abstract}
	Let $\mathcal{C}_{r}^+$ be the set of all balanced signed graphs with underlying graphs $C_r$. A signed graph is called $\mathcal{C}_{r}^+$-free if it contains no subgraph isomorphic to any member of $\mathcal{C}_{r}^+$. In this paper, we characterize the extremal graphs that achieve the maximum index among all $\mathcal{C}_{2k+1}^+$-free unbalanced signed graphs of order $n$ when $k\ge 2$ and $n\ge 187k$.
	
	\smallskip
	
	
\end{abstract}

\section{Introduction}
A \emph{signed graph} $\Gamma=(G,\sigma)$ consists of a simple graph $G=(V(G),E(G))$ together with a signature $\sigma \colon E(G) \to \{+1,-1\}$ on its edge set. The unsigned graph $G$ is called the \emph{underlying graph} of $\Gamma$. In 1946, Heider \cite{Hei-JP-1946} proposed balance theory, which laid the conceptual foundation for the development of signed graphs. Building on Heider's work, Harary \cite{Har-MichM-1953} first formally established signed graph theory, aiming to mathematically formalize balance theory governing group relationships. After nearly three decades, Chaiken \cite{Chaiken-SIAM-1982} and Zaslavsky \cite{Zaslavsky-DAM-1982} independently obtained the matrix-tree theorem for signed graphs. For more information on signed graphs, please refer to \cite{Beck-Zaslavsky-JCTB-2006,Zaslavsky-book-2010,Zaslavsky-EJC-2018,Zaslavsky-JCTB-1987}.

Let $\Gamma$ be a signed graph with the vertex set $V(\Gamma)=\{v_1,v_2,\ldots,v_n\}$. An edge $v_i v_j$ is called a \emph{positive edge} (respectively, \emph{negative edge}) if $\sigma(v_i v_j)=+1$ (respectively, $\sigma(v_i v_j)=-1$). The \emph{adjacency matrix} of $\Gamma$ is defined as $A(\Gamma)=(a_{ij}^{\sigma})$, where $a_{ij}^{\sigma}=\sigma(v_iv_j)$ if $v_i$ and $v_j$ are adjacent, and $a_{ij}^{\sigma}=0$ otherwise. Let
\[
\lambda_1(\Gamma)\ge\lambda_2(\Gamma)\ge\cdots\ge\lambda_n(\Gamma).
\]
denote the eigenvalues of $A(\Gamma)$. The largest eigenvalue $\lambda_1(\Gamma)$ is called the \emph{index} of $\Gamma$, and \[
\rho(\Gamma)=\max\{\lambda_1(\Gamma),-\lambda_n(\Gamma)\}
\]
is called the \emph{spectral radius} of $\Gamma$.
For an unsigned graph $G$, we write $\lambda_1(G)$ for the largest eigenvalue of its adjacency matrix.

Let $\emptyset \ne U\subseteq V(G)$. The operation of changing the signs of all edges between $U$ and $V(G) \setminus U$ is called a \emph{switching}, and is also referred to as \emph{switching $\Gamma$ at a vertex subset $U$}. A signed graph $\Gamma^{\prime}$ is said to be \emph{switching equivalent} to $\Gamma$ if $ \Gamma^{\prime}$ is obtained from $\Gamma$ by a finite sequence of switchings, denoted as $\Gamma\sim \Gamma^{\prime}$; otherwise, we write \(\Gamma\not\sim\Gamma'\). A cycle is called \emph{positive} if the number of its negative edges is even; otherwise, \emph{negative}. A signed graph is called \emph{balanced} if each of its cycles is positive; otherwise, \emph{unbalanced}.

Given a graph $F$, a graph $G$ is said to be \emph{$F$-free} if it contains no subgraph isomorphic to $F$, and the classical Tur\'an problem seeks the maximum number of edges over all $F$-free graphs on $n$ vertices. Let $K_m$ denote the complete graph of order $m$. In 2007, Nikiforov~\cite{Nikiforov} proved that every $K_{r+1}$-free graph $G$ of order $n$ satisfies
\[
\lambda_1(G)\le\lambda_1(T_r(n)),
\]
with equality holding if and only if $G\cong T_r(n)$, where $T_r(n)$ denotes the complete $r$-partite graph on $n$ vertices whose partite sizes differ by at most one. This result is a spectral counterpart of the classical Tur\'an theorem and has motivated extensive research on spectral extremal problems for forbidden subgraphs.

For odd cycles, it is natural to study the corresponding spectral extremal problem under the additional assumption that the graph is non-bipartite. Let $C_m$ denote the cycle of order $m$. In 2021, Guo, Lin and Zhao \cite{Guo-Lin-Zhao-LAA-2021} proved that if $G$ is a non-bipartite $C_5$-free graph of order $n\ge21$, then
\[
\lambda_1(G)\le\lambda_1(B_n),
\]
with equality holding if and only if $G\cong B_n$, where $B_n$ denotes the graph obtained by identifying a vertex of a triangle with a vertex in the smaller part of $K_{\lfloor(n-2)/2\rfloor,\lceil(n-2)/2\rceil}$, and $K_{s,t}$ denotes the \emph{complete bipartite graph} on $s+t$ vertices with vertex bipartition $(X,Y)$ such that $|X|=s$, $|Y|=t$, and all possible edges between $X$ and $Y$ are present. Subsequently, Zhang and Zhao~\cite{Zhang-Zhao-DM-2023} extended this conclusion to $C_{2k+1}$-free non-bipartite graphs for $k\ge3$ and sufficiently large $n$. More recently, Zou, Feng and Li \cite{Zou-Feng-Li-2026} established a general result for non-bipartite graphs without prescribed odd cycles.


Given a family of signed graphs $\mathcal F$, a signed graph is called \emph{$\mathcal F$-free} if it contains no subgraph isomorphic to any member of $\mathcal F$, where the isomorphism preserves edge signs. For $r\ge3$, let $\mathcal C_r^+$ (respectively, $\mathcal C_r^-$) denote the families of all balanced (respectively, unbalanced) signed graphs whose underlying graph is $C_r$. In 2024, Wang, Hou and Li \cite{Wang-Hou-Li-LAA-2024} determined the maximum spectral radius of $\mathcal{C}_{3}^-$-free graphs. In the same year, Wang and Lin \cite{Wang-Lin-DAM-2024} gave the maximum index of $\mathcal{C}_{4}^-$-free graphs. In 2025, Wang, Hou and Huang  \cite{Wang-Hou-Huang-DAM-2025} determined the maximum index of $\mathcal{C}_{2k+1}^-$-free graphs for $3\le k\le \frac{n}{10}-1$. Very recently, Xie, Zhu and Liu~\cite{Xie-Zhu-Liu-AR-2026} determined the maximum index of $\mathcal{C}_{3}^+$-free graphs and characterized the extremal graphs.


A distinction arises between forbidding positive and negative odd cycles when maximizing the index.  Let $-\Gamma$ be the signed graph obtained from $\Gamma$ by reversing the sign of every edge. This operation reverses the sign of every odd cycle. Although this bijection maps signed graphs without negative odd cycles to those without positive odd cycles, the index fails to be preserved, since
\[
\lambda_1(-\Gamma)=-\lambda_n(\Gamma),
\]
which is not equal to $\lambda_1(\Gamma)$ in general. Consequently, extremal results for forbidden negative odd cycles do not directly yield the corresponding maximum-index results for forbidden positive odd cycles.

Let $(G, +)$ denote the signed graph with all positive edges, whose underlying graph is $G$. Let $T_2^*(n)$ denote the signed graph obtained from $\bigl(K_{\lceil n/2\rceil,\lfloor n/2\rfloor},+\bigr)$ by adding a negative edge between two vertices in the part of size $\lceil n/2\rceil$. Clearly, $T_2^*(n)$ contains a negative triangle and hence is unbalanced. Every odd cycle in $T_2^*(n)$ contains its unique negative edge and is therefore negative. Thus, $T_2^*(n)$ contains no positive odd cycle.

In this paper, we determine the maximum index among all unbalanced signed graphs with no positive cycle of a prescribed odd length. Our main result is as follows.

\begin{theorem}\label{C2k+1-free}
	Let $k \ge 2$ and $n \ge 187k$.
	If $\Gamma$ is an unbalanced $\mathcal{C}_{2k+1}^+$-free signed graph of order $n$, then
	\[
	\lambda_1(\Gamma) \le \lambda_1(T_2^*(n)),
	\]
	with equality holding if and only if $\Gamma \sim T_2^*(n)$.
\end{theorem}


\section{Proof of Theorem \ref{C2k+1-free}}\label{Preliminaries}
In this section, we give the proof of Theorem \ref{C2k+1-free}. Before proceeding, we present some necessary results first.

\begin{lemma}\emph{(See \cite[Lemma~2.5]{Sun-Liu-Lan-LAA-2022})}\label{Sun-Liu-Lan-LAA-2022}
	Let $\Gamma$ be a signed graph of order $n$. Then there exists a signed graph $\Gamma^{\prime}$ switching equivalent to $\Gamma$ for which the largest eigenvalue $\lambda_{1}(\Gamma^{\prime})$ admits a nonnegative eigenvector.
\end{lemma}

The following two lemmas show that both the balance property and the eigenvalues of signed graphs are invariant under switching equivalence.
\begin{lemma}\emph{(See \cite[Proposition~3.2]{Zaslavsky-DAM-1982})}\label{Zaslavsky-DAM-1982}
	Two signed graphs on the same underlying graph are switching equivalent if and only if they have the same list of balanced cycles.
\end{lemma}

\begin{lemma}\emph{(See \cite[Lemma~2.1]{Hou-Tang-Wang-AMC-2019})}
	\label{Hou-Tang-Wang-AMC-2019}
	Let $\Gamma_{1}=(G,\sigma_{1})$ and $\Gamma_{2}=(G,\sigma_{2})$ be two signed graphs on the same underlying graph $G$. The following are equivalent:
	
	$\mathrm{(1)}$~$\Gamma_{1}$ and $\Gamma_{2}$ are switching equivalent.
	
	$\mathrm{(2)}$~$A(\Gamma_{1})$ and $A(\Gamma_{2})$ are signature similar.
\end{lemma}

Recall that $B_n$ is the unsigned graph obtained by identifying a vertex of a triangle with a vertex in the smaller part of $K_{\left\lfloor \frac{n-2}{2} \right\rfloor,\left\lceil \frac{n-2}{2} \right\rceil}$.
\begin{lemma}\emph{(See \cite[Theorem~2.1(b)]{Zou-Feng-Li-2026})}\label{Zou-Feng-Li-2026}
	Let $k \ge 2$ and $n \ge 187k$. If $G$ is an $n$-vertex non-bipartite $C_{2k+1}$-free graph, then
	$$\lambda_1(G) \le \lambda_1(B_n),
	$$
	with equality holding if and only if $G \cong B_n$.
\end{lemma}

For integers $a \ge 2$, $b \ge 1$, let $H_{a,b}$ be the signed graph obtained from $(K_{a,b},+)$ by adding a negative edge inside the part of size $a$.

\begin{lemma}\label{Hab}
	Let $n\ge 4 $ be an integer, and let $a\ge 2$, $b\ge 1$ be two integers satisfying $a+b=n$. Then $\lambda_1(H_{a,b})$ is simple and possesses a positive eigenvector;
	\[
	\lambda_1(H_{a,b}) \le \lambda_1(T_2^*(n)),
	\]
	with equality holding if and only if $a = \left\lceil \frac{n}{2} \right\rceil$ and $b = \left\lfloor \frac{n}{2} \right\rfloor$. Moreover,
	\[
	\frac{2\left(\left\lfloor\frac{n^2}{4} \right\rfloor-1\right)}{n} \le \lambda_1(T_2^*(n))< \sqrt{\left\lfloor \frac{n^2}{4} \right\rfloor},\qquad \lambda_1^2(T_2^*(n)) > \left\lfloor \frac{n^2}{4} \right\rfloor-2.
	\]
\end{lemma}

\begin{proof}
	Let $U = \{u_1,u_2,\dots,u_a\}$ and $W = \{w_1,w_2,\dots,w_b\}$ be the two partite sets of $H_{a,b}$, where $u_1u_2$ is the unique negative edge.
	
	If $a \ge 3$, using the vertex partition
	\[
	V_1 = \{u_1,u_2\},\quad V_2 = U \setminus V_1,\quad V_3 = W,
	\]
	we obtain the equitable quotient matrix
	\[
	Q_{a,b} = \begin{pmatrix}
		-1 & 0 & b \\
		0 & 0 & b \\
		2 & a-2 & 0
	\end{pmatrix}.
	\]
	By direct calculation, its characteristic polynomial is
	\[
	f_{a,b}(t) = \det(tI - Q_{a,b}) = t^3 + t^2 - abt - b(a-2).
	\]
	Thus the eigenvalues of $H_{a,b}$ consist of the three roots of $f_{a,b}(t)$,  together with $1$ (multiplicity $1$) and $0$ (multiplicity $n-4$). Indeed, eigenvalues $1$ and $0$ correspond to vectors orthogonal to the constant subspace. More precisely, the vector supported on $\{u_1,u_2\}$ with coordinates $1,-1$ has eigenvalue $1$, and the vectors supported on $V_2$ or $V_3$ whose coordinates sum to zero have eigenvalue $0$.
	
	Note that $Q_{a,b}+I$ is nonnegative and irreducible. By the Perron--Frobenius theorem, the largest eigenvalue of $Q_{a,b}$ is simple and has a positive eigenvector. Furthermore, $f_{a,b}$ has exactly one positive root and
	\[
	f_{a,b}(1) = 2 - 2b(a-1) < 0.
	\]
	Hence, the positive root of $f_{a,b}$ is greater than $1$ and is the largest eigenvalue of $A(H_{a,b})$. Lifting its positive eigenvector from the quotient matrix gives a strictly positive eigenvector of $A(H_{a,b})$.

	If $a=2$, then the equitable quotient matrix is
	\[
	Q_{2,b} = \begin{pmatrix}
		-1 & b \\
		2 & 0
	\end{pmatrix},
	\]
	whose characteristic polynomial is $g_{2,b}(t)=t^2 + t - 2b $. Thus the eigenvalues of $H_{2,b}$ consist of the two roots of $g_{2,b}(t)$,  together with $1$ (multiplicity $1$) and $0$ (multiplicity $b-1$). Similar to the case $a\ge 3$, one can easily verify that the positive root of $g_{2,b}(t)$ is greater than $1$ and is the largest eigenvalue of $A(H_{2,b})$, which is simple and has a positive eigenvector.

	Put $\lambda = \lambda_1(H_{a,b})$, and let $\mathbf{x}$ be a positive eigenvector corresponding to $\lambda$.  By the above equitable vertex partition, set $x_{u_1}=x_{u_2}=z$, $x_{u_i}=u$ for $3\le i\le a$, and $x_{w_j}=v$ for $1\le j\le b$. Then the eigenvalue equation $A(H_{a,b})\mathbf{x}=\lambda\mathbf{x}$ implies that
	\[
	(\lambda + 1)z = bv,\quad \lambda u = bv,\quad \lambda v = 2z + (a-2)u,
	\]
	where all terms involving $u$ are omitted in the case $a=2$. Hence,
	\begin{equation}\label{a-b}
		1 = b\left(\frac{2}{\lambda(\lambda+1)} + \frac{a-2}{\lambda^2}\right),\qquad \lambda^2 = ab - \frac{2b}{\lambda+1}.
	\end{equation}

	Set $a_0 = \lceil n/2\rceil$, $b_0 = \lfloor n/2\rfloor$ and $\theta = \lambda_1(T_2^*(n)) = \lambda_1(H_{a_0,b_0})$. Clearly, $f_{a_0,b_0}(\theta)=0$ and $\theta>1$. Let $b = b_0 + d$, where $d$ is an integer. Then $a=a_0-d$. By direct calculation,
	\[
	\begin{aligned}
		f_{a,b}(\theta) - f_{a_0,b_0}(\theta) &= (a_0 b_0 - ab)(\theta + 1) + 2(b - b_0) \\
		&=
		\begin{cases}
			d\big(d(\theta + 1) + 2\big), & n\text{ is even},\\[0.2cm]
			d\big((d - 1)(\theta + 1) + 2\big), & n\text{ is odd}.
		\end{cases}
	\end{aligned}
	\]
	
	\noindent\emph{Case~1.}~$n$ is even. If $d \ge 1$, then $d(d(\theta + 1) + 2) > 0$. If $d \le -1$, then
	\[
	d(\theta + 1) + 2 \le 1 - \theta < 0,
	\]
	and then $d(d(\theta + 1) + 2) > 0$. Thus $f_{a,b}(\theta) \ge 0$ with equality  holding if and only if $d = 0$.
	
	\noindent\emph{Case~2.}~$n$ is odd. If $d \ge 1$, then $(d-1)(\theta + 1) + 2 \ge 2$, and then $d\big((d - 1)(\theta + 1) + 2\big)>0$. If $d \le -1$, then
	\[
	(d-1)(\theta + 1) + 2 \le -2\theta < 0,
	\]
	and then $d\big((d - 1)(\theta + 1) + 2\big)>0$. Thus $f_{a,b}(\theta) \ge 0$ with equality  holding if and only if $d = 0$.
	
	From the discussion above, we have $$f_{a,b}(\theta) - f_{a_0,b_0}(\theta)\ge 0.$$
	Then
	\[
	\lambda_1(H_{a,b}) \le \theta,
	\]
	with equality holding if and only if $a = a_0$ and $b = b_0$.

	Let $\mathbf{1}$ denote the all-ones column vector of size $n$, and $\mathbf{1}^T$ its transpose. By the Rayleigh-Ritz principle, we obtain
	\begin{equation}\label{a-b-c-121}
		\lambda_1(T_2^*(n)) \ge \frac{\mathbf{1}^T A(T_2^*(n))\mathbf{1}}{\mathbf{1}^T \mathbf{1}} = \frac{2(a_0 b_0-1)}{n} > b_0 - 1.
	\end{equation}
	The last inequality in \eqref{a-b-c-121} follows from
	\[
	2(a_0 b_0-1) - n(b_0 - 1) =
	\begin{cases}
		n-2, & n\text{~is~even},\\[0.2cm]
		3b_0 - 1, & n\text{~is~odd},
	\end{cases}
	\]
	which is positive for $n \ge 4$. By \eqref{a-b} and \eqref{a-b-c-121}, we obtain
	\[
	a_0 b_0 - 2 < a_0 b_0 - \frac{2b_0}{\lambda_1(T_2^*(n)) + 1} = \lambda_1^2(T_2^*(n))< a_0 b_0.
	\]
	
	
	This completes the proof.\qed
\end{proof}


\begin{lemma}\label{Bn}
	Let $n \ge 21$ be an integer. Then
	\[
	\lambda_1(B_n) < \frac{n}{2} - \frac14 < \lambda_1(T_2^*(n)).
	\]
\end{lemma}

\begin{proof}
	Let \(U\) and \(W\) be the two partite sets of the complete bipartite subgraph of \(B_n\), where
	$$
	|U|=\left\lfloor\frac{n-2}{2}\right\rfloor \quad\text{and}\quad |W|=\left\lceil\frac{n-2}{2}\right\rceil.
	$$
	Assume that \(u\in U\) is the vertex identified with a vertex of the triangle, and let \(v_1,v_2\) be the other two vertices of the triangle. Set
	$$
	L=\frac n2-\frac{1}{4}.
	$$
	Let \(\mathbf z=(z_v)_{v\in V(B_n)}\) be the positive vector defined by
	$$
	z_v= \begin{cases}
		1, & v\in U\cup W,\\[0.2cm]
		\dfrac{1}{8}, & v\in\{v_1,v_2\}.
	\end{cases}
	$$
	By direct calculation, we have
	$$
	\bigl(A(B_n)\mathbf z\bigr)_u =|W|+\frac14 =\left\lceil\frac{n-2}{2}\right\rceil+\frac14 \le \frac n2-\frac14 =Lz_u.
	$$
	For any \(v\in(U\cup W)\setminus\{u\}\), we have
	$$
	\bigl(A(B_n)\mathbf z\bigr)_v =d_{B_n}(v) \le\left\lceil\frac{n-2}{2}\right\rceil <Lz_v.
	$$
	For \(i=1,2\), we have
	$$
	\bigl(A(B_n)\mathbf z\bigr)_{v_i} =1+\frac18 =9z_{v_i} <Lz_{v_i},
	$$
	since \(n\ge21\). Hence,
	$$
	A(B_n)\mathbf z\le L\mathbf z,
	$$
	with strict inequality in at least one coordinate.
	
	Since \(B_n\) is connected, let \(\mathbf x\) be a positive unit eigenvector of \(A(B_n)\) corresponding to \(\lambda_1(B_n)\). Then
	$$
	\lambda_1(B_n)\mathbf x^{T}\mathbf z =\mathbf x^{T}A(B_n)\mathbf z <L\mathbf x^{T}\mathbf z.
	$$
	Note that \(\mathbf x^{T}\mathbf z>0\). Thus, we have
	$$
	\lambda_1(B_n)<L=\frac n2-\frac14.
	$$
	On the other hand, by Lemma \ref{Hab}, we have
	$$
	\lambda_1(T_2^*(n))\ge\frac{2\left(\left\lfloor \frac{n^2}{4} \right\rfloor-1\right)}{n} =
	\begin{cases}
		\displaystyle\frac n2-\frac2n, & n\text{ is even},\\[0.3cm]
		\displaystyle\frac n2-\frac{5}{2n}, & n\text{ is odd}.
	\end{cases}
	$$
	Recall that \(n\ge21\). Then $\lambda_1(T_2^*(n))>  \frac{n}{2}-\frac{1}{4}$ in either case. Therefore,
	$$
	\lambda_1(B_n)<\frac n2-\frac14<\lambda_1(T_2^*(n)).
	$$
	
	This completes the proof.\qed
	
\end{proof}

For integers $a, b \ge 2$,  let $F_{a,b}$ be the signed complete bipartite graph with exactly one negative edge.

\begin{lemma}\label{a+b=n}
	Let $a, b \ge 2$ be two integers, and $n=a+b$. Then
	\[
	\lambda_1(F_{a,b}) < \lambda_1(T_2^*(n)),
	\]
\end{lemma}

\begin{proof}
	Let $U = \{u_1,u_2,\dots,u_a\}$ and $W = \{w_1,w_2,\dots,w_b\}$ be the two partite sets of $F_{a,b}$, where $u_1w_1$ is the unique negative edge. Using the vertex partition
	\[
	V_1 = \{u_1\},\quad V_2 = U\setminus\{u_1\},\quad V_3 = \{w_1\},\quad V_4 = W\setminus\{w_1\},
	\]
	we obtain the equitable quotient matrix
	\[
	Q(F_{a,b}) =
	\begin{pmatrix}
		0 & 0 & -1 & b-1 \\
		0 & 0 & 1 & b-1 \\
		-1 & a-1 & 0 & 0 \\
		1 & a-1 & 0 & 0
	\end{pmatrix}.
	\]
	By direct calculation, the characteristic polynomial of $Q(F_{a,b})$ is
	\[
	g_{a,b}(t) = t^4 - abt^2 + 4(a-1)(b-1).
	\]
	Thus the eigenvalues of $F_{a,b}$ consist of the four roots of $g_{a,b}(t)$,  together with $0$ (multiplicity $n-4$). Indeed, eigenvalue $0$ corresponds to vectors orthogonal to the constant subspace.  Hence,
	\[
	\lambda_1(F_{a,b}) = \lambda_{1}(Q(F_{a,b})).
	\]
	Then $s= \lambda_1(F_{a,b})^2$ is the larger root of
	\begin{equation}\label{Fab}
		h_{a,b}(x) = x^2 - abx + 4(a-1)(b-1),
	\end{equation}
	that is,
	$$
	s=\frac{ab}{2}+ \sqrt{\frac{a^2 b^2}{4}-4(a-1)(b-1)}.
	$$
	Clearly, $0 < s < ab$.

	\noindent\emph{Case~1.}~$n \ge 8$. If $ab \le \left\lfloor \frac{n^2}{4} \right\rfloor - 2$, by Lemma \ref{Hab}, we have
	\[
	\lambda_1(F_{a,b})^2=s < ab \le \left\lfloor \frac{n^2}{4} \right\rfloor- 2 < \lambda_1^2(T_2^*(n)),
	\]
	which implies that $\lambda_1(F_{a,b}) < \lambda_1(T_2^*(n))$.
	
	If $ab \ge \left\lfloor \frac{n^2}{4} \right\rfloor - 1$,  we have
	$$
	(a-b)^2 = n^2 - 4ab\le n^2- 4\left\lfloor \frac{n^2}{4} \right\rfloor +4,
	$$
	since $a+b=n$. Then $|a-b| \le 2$. Note that $a+b=n \ge 8$. Then
	\[
	(a-2)(b-2) > 2,
	\]
	which implies that
	\[
	4(a-1)(b-1) - 2ab = 2\big((a-2)(b-2)-2\big) > 0.
	\]
	By Lemma \ref{Hab}, \eqref{Fab} and $s < ab$, we have
	\begin{align*}
		s &= ab - \frac{4(a-1)(b-1)}{s} \\
		&< ab - \frac{4(a-1)(b-1)}{ab} \\
		&< ab - 2  \\
		&\le \left\lfloor \frac{n^2}{4} \right\rfloor - 2\\
		&< \lambda_1^2(T_2^*(n)),
	\end{align*}
	which implies that $\lambda_1(F_{a,b}) < \lambda_1(T_2^*(n))$.

	\noindent\emph{Case~2.}~$4 \le n \le 7$. Since $F_{a,b}\cong F_{b,a}$, without loss of generality, assume that $a \le b$.
	
	If $n=4$, then $ a=b=2$. By \eqref{Fab} and Lemma \ref{Hab}, we have $s=2 < \lambda_1^2(T_2^*(4))$.
	
	If $n=5$, then $ a=2$ and $b=3$. By \eqref{Fab} and Lemma \ref{Hab}, we have $s =4 < \lambda_1^2(T_2^*(5))$.
	
	If $n=6$, then $a=2$ and $b=4$, or $a=b=3$. By \eqref{Fab}, $s=6$ or $\frac{9+\sqrt{17}}{2}$. By Lemma \ref{Hab}, we have
	\[
	s  < 7 = \left\lfloor \frac{6^2}{4} \right\rfloor - 2 < \lambda_1^2(T_2^*(6))
	\]
	in either case.
	
	If $n=7$, then $a=2$ and $b=5$, or $a=3$ and $b=4$. By \eqref{Fab}, $s=8$ or $6+2\sqrt{3}$. By Lemma \ref{Hab}, we have
	\[
	s < 10 = \left\lfloor \frac{7^2}{4} \right\rfloor- 2 < \lambda_1^2(T_2^*(7))
	\]
	in either case.
	
	Thus $\lambda_1(F_{a,b}) < \lambda_1(T_2^*(n))$ in each case.
	
	This completes the proof.\qed
	
\end{proof}


A \emph{signed subgraph} of a signed graph is a subgraph whose edge
signs are inherited from the original signed graph. Let $\Gamma_+$ (respectively, $\Gamma_-$) denote the spanning
signed subgraph of a signed graph $\Gamma$ consisting of all its positive
(respectively, negative) edges. Thus,
$V(\Gamma_+)=V(\Gamma_-)=V(\Gamma)$, with isolated vertices retained
in both subgraphs.

\begin{lemma}\label{Gamma+}
	Let $n \ge 4$ be an integer, and let $\Gamma$ be an  unbalanced signed graph of order $n$ with a nonnegative unit eigenvector $\mathbf{x}$ for its index. If $\Gamma_{+}$ is bipartite, then
	\[
	\lambda_1(\Gamma) \le \lambda_1(T_2^*(n)),
	\]
	with equality holding if and only if $\Gamma \sim T_2^*(n)$.
\end{lemma}
\begin{proof}
	Since $\Gamma$ is unbalanced, $\Gamma$ contains at least one negative edge.	Moreover, $\Gamma_{+}$ contains at least one edge; otherwise $\mathbf{x}^\top A(\Gamma)\mathbf{x} \le 0$, a contradiction.
	
	Let \((U,W)\) be a bipartition of the vertex set of $\Gamma_{+}$. Suppose that a negative edge $e$ lies within one part, say $U$. Using the bipartition \((U,W)\), we construct a new signed graph $\Sigma:= H_{|U|,|W|}$. Clearly, $A(\Sigma)-A(\Gamma)$ is entrywise nonnegative. Thus, by Lemma \ref{Hab} and the Rayleigh-Ritz principle, we have
	\[
	\lambda_1(\Gamma) = \mathbf{x}^T A(\Gamma)\mathbf{x} \le \mathbf{x}^T A(\Sigma)\mathbf{x} \le \lambda_1(\Sigma) \le \lambda_1(T_2^*(n)).
	\]
	If equality holds, Lemma \ref{Hab} and the Rayleigh-Ritz principle imply the following two facts: $\mathrm{(i)}~\Sigma=T_2^*(n)$; $\mathrm{(ii)}~\mathbf{x}$ is also an eigenvector of $A(\Sigma)$ corresponding to $\lambda_1(\Sigma)$. Hence, all coordinates of $\mathbf{x}$ are positive. Then, we have $A(\Sigma)-A(\Gamma)=\mathbf{0}$. Hence, $\Gamma=\Sigma= T_2^*(n)$.
	
	If no edges lie inside $U$ or $W$, then the underlying graph of $\Gamma$ is bipartite. If one part contained only one vertex, $\Gamma$ would be a subgraph of a star and hence balanced, a contradiction. Therefore both parts have at least two vertices. Using the bipartition \((U,W)\), we construct a new signed graph $\Sigma_1 := F_{|U|,|W|}$. Clearly, $A(\Sigma_1)-A(\Gamma)$ is entrywise nonnegative. By Lemma \ref{a+b=n} and the Rayleigh-Ritz principle, we have
	$$
	\lambda_1(\Gamma)  = \mathbf{x}^T A(\Gamma)\mathbf{x} \le \mathbf{x}^T A(\Sigma_1)\mathbf{x} \le \lambda_1(\Sigma_1) < \lambda_1(T_2^*(n)).
	$$
	
	This completes the proof.\qed
\end{proof}

\begin{Tproof}\textbf{~of~Theorem~\ref{C2k+1-free}.}~~Let $\Gamma=(G,\sigma)$ be a signed graph having the maximum index over all $\mathcal{C}^{+}_{2k+1}$-free unbalanced signed graphs of order $n$. By Lemma \ref{Sun-Liu-Lan-LAA-2022}, let $\Gamma^{\prime}=(G,\sigma^{\prime})$ be the signed graph switching equivalent to $\Gamma$ for which $\lambda_{1}(\Gamma^{\prime})$ admits a nonnegative eigenvector. Let $\mathbf{x}$ be a nonnegative unit eigenvector corresponding to $\lambda_{1}(\Gamma^{\prime})$. By Lemmas \ref{Zaslavsky-DAM-1982} and \ref{Hou-Tang-Wang-AMC-2019}, $\Gamma^{\prime}=(G,\sigma^{\prime})$ is also a $\mathcal{C}^{+}_{2k+1}$-free unbalanced signed graph with the maximum index. Clearly, $\Gamma^{\prime}_+$ is $C_{2k+1}$-free.
	
	If $\Gamma^{\prime}_+$ is non-bipartite, by Lemmas \ref{Zou-Feng-Li-2026} and \ref{Bn}, we have
	\[
	\lambda_1(\Gamma) = \lambda_1(\Gamma^{\prime})=\mathbf{x}^{T}A(\Gamma^{\prime})\mathbf{x}\le \mathbf{x}^{T}A(\Gamma^{\prime}_+)\mathbf{x} \le \lambda_1(\Gamma^{\prime}_+) \le \lambda_1(B_n) < \lambda_1(T_2^*(n)).
	\]
	
	If $\Gamma^{\prime}_+$ is bipartite, by Lemma \ref{Gamma+}, we obtain
	\[
	\lambda_1(\Gamma) \le \lambda_1(T_2^*(n)),
	\]
	with equality holding if and only if $\Gamma \sim T_2^*(n)$.
	
	This completes the proof.\qed
\end{Tproof}

\end{document}